\documentclass[a4paper,12pt]{article}
\usepackage{amssymb,amsthm,amsmath,amsfonts}
\usepackage{graphicx}
\usepackage{cite}
\usepackage{color}

\newtheorem{theorem}{Theorem}[section]

\newtheorem{corollary}[theorem]{Corollary}
\newtheorem{proposition}[theorem]{Proposition}

\theoremstyle{definition}
\newtheorem{definition}[theorem]{Definition}
\newtheorem{example}[theorem]{Example}
\theoremstyle{remark}
\newtheorem{remark}{Remark}

\begin{document}
\title{Fractional slice regular functions of a quaternionic variable}
\small{
\author {Jos\'e Oscar Gonz\'alez-Cervantes$^{(1)}$, Juan Bory-Reyes$^{(2)\footnote{corresponding author}}$,\\and\\ Irene Sabadini$^{(3)}$}
\vskip 1truecm
\date{\small $^{(1)}$ Departamento de Matem\'aticas, ESFM-Instituto Polit\'ecnico Nacional. 07338, Ciudad M\'exico, M\'exico\\ Email: jogc200678@gmail.com\\$^{(2)}$ {SEPI, ESIME-Zacatenco-Instituto Polit\'ecnico Nacional. 07338, Ciudad M\'exico, M\'exico}\\ Email: juanboryreyes@yahoo.com \\$^{(3)}$ Politecnico di Milano, Dipartimento di Matematica, Via E. Bonardi 9, 20133
Milano, Italy\\ Email: irene.sabadini@polimi.it
}
\maketitle
\begin{abstract}
The theory of slice regular functions of a quaternionic variable on the unit ball of the quaternions was introduced by Gentili and Struppa in 2006 and nowadays it is a well established function theory, especially in view of its applications to operator theory.

In this paper, we introduce the notion of fractional slice regular functions of a quaternionic variable defined as null-solutions of a fractional Cauchy-Riemann operators. We present a fractional Cauchy-Riemann operator in the sense of Riemann-Liouville and then in the sense of Caputo, with orders associated to an element of $(0,1)\times \mathbb R \times (0,1)\times \mathbb R$ for some axially symmetric slice domains which are new in the literature. We prove a version of the representation theorem, of the splitting lemma and we discuss a series expansion.

\end{abstract}

\section{Introduction}
The main purpose of this paper is to combine the fractional calculus with the theory of slice regular functions of a quaternionic variable.

Fractional calculus, albeit a synonym of fractional integrals and derivatives is nowadays a well-studied tool in several branches of science and engineering, see \cite{KST, OS, O, P, MR, SKM}. Fractional integrals and derivatives are nearly as old as the conventional calculus, in fact in 1695 de L'Hopital and Leibnitz mentioned already the notion of half derivative, see \cite{Le}. However a rigorous investigation was firstly carried out in the 19th century by Liouville in a series of reports of the Ecole Polytechnique in Paris in 1832-1837, see for instance \cite{L}. A decade later, Riemann in 1847 \cite{R}, described an integral-based Riemann-Liouville fractional integral operator \cite{CC}. For a brief history of the fractional calculus we refer the reader to \cite{Ro}.

The development of a fractional hyperholomorphic function theory (for functions with values in the quaternions or more in general in a Clifford algebra) is a recent field of research, see \cite{GB-1, GB-2, CTOP, DM, FRV, KV, PBBB, V} and the references therein. In particular, the interest for considering fractional Laplace and Dirac Operator is addressed in \cite{Ba, Be, FV1, FV2}.

However, in the hypercomplex setting one can consider another framework, namely that of slice hyperholomorphic functions. These functions were introduced in 2006 by Gentili and Struppa, inspired by a paper by Cullen, in the case of quaternionic valued functions defined on the unit ball  $\mathbb B^4(0,1)$ of the quaternions, identified with $\mathbb R^4$. Functions slice regular are defined as functions whose restriction to any complex plane is in the kernel of the corresponding Cauchy-Riemann operator.  The theory was then generalized from $\mathbb B^4(0,1)$ to more general domains, the so-called axially symmetric slice domains, see \cite{CSS,GenSS} and the references therein.

A great impulse for the theory   was given by the applications of this class of functions to operator theory, see \cite{CSS,ACS}.  More recently, the theory has been extended to even more general domains dropping the condition of axial symmetry, see \cite{GSto}, or even in greater generality using a new topology, the so-called slice topology, instead of the Euclidean one, see \cite{DJRS, DRS}. For a general account of the theory we refer the reader to \cite{CSS,CSS2,GenSS},
 and the references therein. It is interesting to note that fractional powers of quaternionic operators have been under study since the paper \cite{CG1} and the theory has been further developed also in the case of vector operators, see \cite{CG2}.

In this paper, we introduce a notion of fractional slice regular functions of a quaternionic variable on some axially symmetric slice domains. These functions are such that their restrictions to any complex plane are null-solutions of a suitably defined fractional Cauchy-Riemann operator. We shall consider two fractional Cauchy-Riemann operators: firstly, that one in the sense of Riemann-Liouville and, secondly, in the sense of Caputo. In both cases, the operators are with orders associated with an element of $(0,1)\times \mathbb R \times (0,1)\times \mathbb R$. We then prove a version of the representation formula, of the splitting lemma in this framework.

The paper is organized as follows. After this introduction, in Section 2 we briefly recall basic concepts and facts on fractional derivatives of complex order in the sense of Riemann-Liouville and of Caputo. Section 3 contains a few essential definitions and results about the quaternionic slice regular functions. Section 4 contains the notion of fractional Riemann-Liouville slice regular function and some results among which a splitting lemma and a representation formula valid for this class of functions. Section 5 discusses the case of fractional Caputo slice regular functions.

\section{Riemann-Liouville and Caputo fractional derivatives and integrals}\label{1}
In this section we provide a brief background of the necessary material on the two most popular definitions of fractional derivatives: Riemann-Liouville and Caputo fractional derivatives and integrals. We refer the reader to e.g. \cite{MR} for more details.

Set $a, b \in \mathbb R$    {and  $\alpha \in \mathbb C$ such that $0<\Re(\alpha)< 1$} and $a < b$. We do not discuss here the properties that a function $f$ needs to satisfy in order that the material in this section is valid and we shall work by requiring some necessary hypotheses (not necessarily the most general ones). Thus we suppose that $f \in L^1([a, b], \mathbb R)$ and that it is identically equal to zero outside the interval $[a, b]$.

The left and the right Riemann-Liouville integrals of order $\alpha$ of $f$ are defined, respectively, as follows:
$$({\bf I}_{a^+}^{\alpha} f)(x) := \frac{1}{\Gamma(\alpha)} \int_a^x \frac{f(\tau)}{(x-\tau)^{1-\alpha}} d\tau, \ \textrm{with} \ x > a$$
and
$$({\bf I}_{b^-}^{\alpha} f)(x) := \frac{1}{\Gamma(\alpha)} \int_x^b \frac{f(\tau)}{(\tau-x)^{1-\alpha}} d\tau, \ \textrm{with} \ x < b.$$
\begin{definition}
Let $f\in AC^1([a, b], \mathbb R)$ denote the class of all continuously differentiable functions and absolutely continuous on $[a, b]$. The fractional derivatives in the Riemann-Liouville sense, on the left and on the right, are defined by
\begin{align}\label{FracDer}
({_{RL}}D _{a^+}^{\alpha} f)(x):= \frac{d}{dx} \left[ ({\bf I}_{a^+}^{1-\alpha} f)(x)\right]
\end{align}
and
\begin{align} \label{FracDer1}
({_{RL}}D _{b^-}^{\alpha} f)(x):= (-1)\frac{d}{dx}\left[({\bf I}_{b^-}^{1-\alpha}f)(x)\right],
\end{align}
respectively.
\end{definition}
\begin{remark}
It is worthwhile noting that the derivatives in (\ref{FracDer}), (\ref{FracDer1}) do exist for $f\in AC^1([a, b], \mathbb R)$.
\end{remark}
\begin{remark}
The Fundamental Theorem for the Riemann-Liouville fractional calculus, \cite{CG}, shows that
\begin{align}\label{FundTheorem}
({_{RL}}D_{a^+}^{\alpha} {\bf I}_{a^+}^{\alpha}f)(x)=f(x) \\
 \ ({_{RL}}D _{b^-}^{\alpha}  {\bf I}_{b^-}^{\alpha} f)(x) = f(x).
\end{align}
\end{remark}
\begin{remark}
Let us mention an important feature of the fractional Riemann-Liouville derivative, namely the violation of the Leibniz rule. This is a characteristic property of the derivatives of non-integer orders, see \cite{Tarasov}.
\end{remark}
The following result is from \cite[pag. 1835]{VTRMB} and illustrate how the Riemann-Liouville fractional derivative acts on some special polynomials:
\begin{proposition}\label{Prop2.2} We have the following identities
\begin{align}\label{cte}
( {_{RL}}D _{a^+}^{\alpha} 1)(x)= &  \frac{(x-a)^{-\alpha}}{\Gamma[1-\alpha]}, \nonumber \\
  {_{RL}}D _{a^+}^{\alpha} (x-a)^{\beta} = &  \frac{\Gamma(\beta+1) (x-a)^{\beta -\alpha}}{\Gamma[\beta+1-\alpha]}, \quad \Re(\beta)>  -1
 \end{align}
for all $x\in [a, b]$. .
\end{proposition}
The second notion of fractional derivative that we consider in this paper is the following:
\begin{definition}
Let $f\in AC^1([a, b], \mathbb R)$.
The fractional derivatives in the Caputo sense, on the left and on the right,  are defined by:
\begin{align}\label{FracDerCaputo}
({_{\mathcal C} }D _{a^+}^{\alpha} f)(x):= ({\bf I}_{a^+}^{1-\alpha}   \frac{d f}{dx} )   (x)
\end{align}
and
\begin{align} \label{FracDer1Caputo}
({}_{\mathcal C}D _{b^-}^{\alpha} f)(x):= (-1)({\bf I}_{b^-}^{1-\alpha}\frac{df}{dx})(x),
\end{align}
respectively.
\end{definition}
\begin{remark}
As in the case of Riemann-Liouville fractional derivatives, from the Fundamental Theorem we have (see \cite{CG}):
\begin{align}\label{FundTheoremCaputo}
({}_{\mathcal C} D_{a^+}^{\alpha} {\bf I}_{a^+}^{\alpha}f)(x)=f(x) \\  ({}_{\mathcal C} D _{b^-}^{\alpha}  {\bf I}_{b^-}^{\alpha} f)(x) = f(x).
\end{align}
\end{remark}
The link between the two fractional derivatives that we have introduced is the following (see  (53) in \cite[Remark 3]{VTRMB}):
\begin{proposition} We have the following identities
\begin{align}\label{RLandC}
{}_{\mathcal C} D _{a^+}^{\alpha} f (x) = {_{RL}} D _{a^+}^{\alpha} (f- f(a)) (x),  \  \forall x\in [a,b],
\end{align}
for $0< \Re (\alpha) <1$.
\end{proposition}
\section{Preliminaries on slice regular functions of a quaternionic variable}
A real quaternion is an element of the form $q=x_0  + x_{1} {e_1} +x_{2} e_2 + x_{3} e_3$ where $x_0, x_1, x_2, x_3\in\mathbb R$ and  the imaginary units $e_1$, $e_2$, $e_3$ satisfy: $e_1^2=e_2^2=e_3^2=-1$,  $e_1e_2=-e_2e_1=e_3$, $e_2e_3=-e_3e_2=e_1$, $e_3e_1=-e_1e_3=e_2$.  Quaternions form a skew-field denoted by  $\mathbb H$. The sets $\{e_1,e_2,e_3\}$ and $\{1,e_1,e_2,e_3\}$ are called  the standard basis of $\mathbb R^3$ and   $\mathbb H$, respectively.

The vector part of $q\in \mathbb H$ is by definition, ${\bf{q}}:= x_{1} {e_1} +x_{2} e_2 + x_{3} e_3$ while its real part is $q_0:=x_0$.
The quaternionic conjugation of $q$, denoted by $\bar q$ is defined by $\bar q=:q_0-{\bf q} $ and norm the $q\in\mathbb H$ is given by
$$\|q\|:=\sqrt{x_0^2 +x_1^2+x_2^3+x_3^2}= \sqrt{q\bar q} = \sqrt{\bar q  q}.$$

The unit open ball in $\mathbb H$ is denoted by $\mathbb B^4(0,1):=\{q \in \mathbb H\ \mid \ \|q\|<1 \}$ and  the unit  spheres   in  $\mathbb R^3$, identified with the set of the quaternions which are purely imaginary,  and in  $\mathbb H$ are $\mathbb{S}^2:=\{{\bf q}\in\mathbb R^3  \mid \|{\bf q}\| =1  \}$    and  $\mathbb{S}^3:=\{ {  q}\in\mathbb H \mid  \|{  q}\| =1\}$, respectively.

The quaternionic structure allows us to see that every ${\bf i}\in \mathbb{S}^2$ satisfies the property ${\bf i}^{2}=-1$ and that the set $\mathbb{C}({\bf i}):=\{x+{\bf i}y; \ |\ x,y\in\mathbb{R}\}\cong \mathbb C$ as fields.

Note that each $q\in \mathbb H \setminus \mathbb R$ can be written as $x+ {\bf i}_q y $ where  $x,y\in \mathbb R$ and ${\bf i}_q:=\|  {\bf q}\|^{-1}{\bf q}\in \mathbb S^2$ if ${\bf q}$ is not the zero vector, for otherwise we have $q=x\in \mathbb R$. This observation leads to the simple but crucial fact that $\mathbb H$ can be seen as union of the complex planes $\mathbb{C}({\bf i})$, i.e.
$$
\mathbb H=\bigcup_{{\bf i}\in\mathbb{S}^2} \mathbb{C}({\bf i}).
$$
This observation leads to the following definition:
\begin{definition}
Let $\Omega\subset\mathbb H$ be an open set. A real differentiable function $f:\Omega\to \mathbb{H}$  is called left slice regular function, or slice regular function on $\Omega$,
if
\begin{align*}
\overline{\partial}_{{\bf i}}f\mid_{_{\Omega\cap \mathbb C({\bf i})}}:=\frac{1}{2}\left (\frac{\partial}{\partial x}+{\bf i} \frac{\partial}{\partial y}\right )f\mid_{_{\Omega\cap \mathbb C({\bf i})}}=0  \textrm{ on  $\Omega\cap \mathbb C({\bf i})$,}
\end{align*}
for all ${\bf i}\in \mathbb{S}^2$ and its derivative is $f'=\displaystyle {\partial}_{{\bf i}}f\mid_{_{\Omega\cap \mathbb C({\bf i})}} = \frac{\partial}{\partial x} f\mid_{_{\Omega\cap \mathbb C({\bf i})}}= \partial_xf\mid_{_{\Omega\cap \mathbb C({\bf i})}}$.
\end{definition}
The quaternionic right linear space of the  slice regular functions on $\Omega$ is denoted by $\mathcal{SR}(\Omega)$.
\\
We note that due to the noncommutativity of the quaternions one may define a notion of right slice regular functions by requiring
\begin{align*}
f\mid_{_{\Omega\cap \mathbb C({\bf i})}}\overline{\partial}_{{\bf i}}:=\frac{1}{2}\left (\frac{\partial}{\partial x}f\mid_{_{\Omega\cap \mathbb C({\bf i})}}+ \frac{\partial}{\partial y}f\mid_{_{\Omega\cap \mathbb C({\bf i})}}{\bf i}\right )=0  \textrm{ on  $\Omega\cap \mathbb C({\bf i})$,}
\end{align*}
for all ${\bf i}\in \mathbb{S}^2$. The function theory is different but fully equivalent to that on left slice regular functions.
\begin{definition}
A set $\Omega\subset\mathbb H$ is called axially symmetric if $x+{\bf i}y \in \Omega$ with $x,y\in\mathbb R$   then $\{x+{\bf j}y \ \mid  \ {\bf j}\in\mathbb{S}^2\}\subset \Omega$ and $\Omega\cap \mathbb R\neq \emptyset$.
A domain $\Omega\subset\mathbb H$ is a slice domain, or s-domain, if $\Omega_{\bf i} := \Omega\cap \mathbb C({\bf i})$ is a domain in $\mathbb C({\bf i})$ for all ${\bf i}\in\mathbb S^2$.
\end{definition}

The following results are well-known and form the basis for the function theory of slice regular functions for which we refer the reader to e.g.  \cite{CSS, GenSS}.
\begin{theorem}\label{1.3}
Let $\Omega \subset\mathbb{H}$ be an axially symmetric s-domain and  $f\in\mathcal{SR}(\Omega)$.

1. (Splitting Lemma) For every ${\bf i},{\bf j}\in \mathbb{S}$, orthogonal  to  each other, there exist $F,G\in Hol(\Omega_{\bf i})$, holomorphic functions from  $\Omega_{\bf i}$ to $\mathbb{C}({{\bf i}})$ such that
$$f_{\mid_{\Omega_{\bf i}}} = F + G{\bf j}\  \mbox{on}\  \Omega_{\bf i},$$
see \cite{CSS}.
%\item

2. (Representation Formula) For every  $q=x+{\bf i}_q y \in \Omega$ with $x,y\in\mathbb R$ and  ${\bf i}_q \in \mathbb S^2$ one has that
\begin{align}\label{RepreForm}
f(x+{\bf i}_q y) = \frac {1}{2}[   f(x+{\bf i}y)+ f(x-{\bf i}y)]
+ \frac {1}{2} {\bf i}_q {\bf i}[ f(x-{\bf i}y)- f(x+{\bf i}y)],
\end{align}
for all ${\bf i}\in \mathbb S^2$.
%\end{enumerate}
\end{theorem}
The class of slice regular functions contains polynomials in the quaternionic variable $q$ and convergent power series centered at $0$ or at a real point. Also the converse is true as it was shown in the original paper \cite{GenS2}:
\begin{theorem}
 Set $f\in \mathcal {SR} (\mathbb B^4(0,1))$  there exists a sequence of quaternions $(a_n)$ such that $f(q)= \sum_{n=0}^{\infty} q^n a_n$ for all $q\in\mathbb B^4(0,1)$.
\end{theorem}
Finally, we recall that slice regular functions can be represented via a Cauchy integral formula with slice regular kernel:
\begin{theorem}
Let $\Omega \subseteq \mathbb{H}$ be an axially symmetric s-domain  such that
$\partial \Omega_{\bf i}$ is union of a finite number of
continuously differentiable Jordan curves, for every ${\bf i}\in\mathbb{S}^2$. Let $f$ be
a slice regular function on an open $U$ set containing $\overline{\Omega}$ and, for any ${\bf i}\in \mathbb{S}^2$,  set  $d\zeta_{{\bf i}}:=-{\bf i}d\zeta$.
Then for every $q\in \Omega$ we have:
\begin{equation}\label{integral}
 f(q)=\frac{1}{2 \pi}\int_{\partial \Omega_{{\bf i}}} S^{-1}(\zeta,q) d\zeta_{{\bf i}} f(\zeta),
\end{equation}
where $S^{-1}(\zeta, q)=-(q^2-2{\rm Re}(\zeta)q+|\zeta|^2)^{-1}
 (q-\overline{\zeta})$ is the slice regular Cauchy kernel.
Moreover,
the value of the integral depends neither on $U$ nor on the  imaginary unit
${\bf i}\in\mathbb{S}^2$.
\end{theorem}

\section{Fractional slice regular functions in the Riemann-Liouville sense}\label{2}

Let  $a, b, c\in\mathbb R$ with $a<b$, $c>0$ and let us consider the special class of axially symmetric s-domains defined by
$$S_{a , b,c }=\{ q=  x+ {\bf i}_q y \in \mathbb H  \  \mid  \ x \in [ a, b] ,  \ y \in [0,c], \  {\bf i}_q\in \mathbb S^2 \}.$$
Given ${\bf i}\in \mathbb S^2$, the intersection of $S_{a , b,c }$ with $\mathbb C({\bf i})$ is denoted by
$$S_{a , b,c,{\bf i} }=\{    x+ {\bf i} y \in \mathbb H  \  \mid  \ x \in [ a, b] ,  \ y \in [0,c] \}.$$
Note that $S_{a , b,c,{\bf i} }\subset \mathbb C^+({\bf i})$.
Fix $u\in [a,b]$ and $v\in [0,c]$ and let $\alpha:=(\alpha_0,\alpha_1)$ and $\beta:=(\beta_0,\beta_1)$ be pairs in $(0,1)\times\mathbb R$.
For $f \in C(S_{a , b,c }, \mathbb H)$ let us define
{\footnotesize
\begin{align*}
[{\bf I}_{a^+}^{1-\alpha_0- {\bf i} \alpha_1} f\mid_{S_{a , b,c,{\bf i} }}](x+{\bf i}v) := & \frac{1}{\Gamma(1-\alpha_0- {\bf i} \alpha_1)} \int_a^x \frac{f\mid_{S_{a , b,c,{\bf i} }}(\tau + {\bf i} v)}{(x-\tau)^{\alpha_0+ {\bf i} \alpha_1}} d\tau, \ \textrm{with} \ x > a, \\
[{\bf I}_{0^+}^{1-\beta_0- {\bf i} \beta_1} f\mid_{S_{a , b, c, {\bf i} }}](u+{\bf i}y) :=& \frac{1}{\Gamma(1-\beta_0- {\bf i} \beta_1)} \int_{0}^y \frac{f\mid_{S_{a , b,c, {\bf i} }}(u + {\bf i}\tau )}{(y-\tau)^{ \beta_0+ {\bf i} \beta_1}} d\tau, \ \textrm{with} \ y > 0 ,\\
[{\bf I}_{b^-}^{1-\alpha_0- {\bf i} \alpha_1} f\mid_{S_{a , b,c,{\bf i} }}](x+{\bf i}v) := & \frac{1}{\Gamma(1-\alpha_0- {\bf i} \alpha_1)} \int_x^b \frac{f\mid_{S_{a , b,c,{\bf i} }}(\tau + {\bf i} v)}{(x-\tau)^{ \alpha_0+ {\bf i} \alpha_1}} d\tau, \ \textrm{with} \ x < b, \\
[{\bf I}_{c^-}^{1-\beta_0- {\bf i} \beta_1} f\mid_{S_{a , b, c, {\bf i} }}](u+{\bf i}y) :=& \frac{1}{\Gamma(1-\beta_0- {\bf i} \beta_1)} \int_{y}^c \frac{f\mid_{S_{a , b,c, {\bf i} }}(u + {\bf i}\tau )}{(y-\tau)^{ \beta_0+ {\bf i} \beta_1}} d\tau, \ \textrm{with} \ y <  c .
\end{align*}
}

Let us denote by $AC^1(S_{a , b,c }, \mathbb H)$ the class of $\mathbb H$-valued continuously differentiable functions $f$ such that the mappings $x\mapsto  f\mid_{S_{a , b, c, {\bf i} }}(x+{\bf i}v)$ and $y \mapsto  f\mid_{S_{a , b, c, {\bf i} }}(u+{\bf i}y)$ are absolutely continuous for all $x\in[a,b]$, $y\in [0,c]$, and ${\bf i} \in \mathbb S^2$.
\begin{definition}\label{FDQ}
The fractional derivatives on the left and on the right in the Riemann-Liouville sense of $f\in AC^1(S_{a , b,c }, \mathbb H)$  associated with the slice $\mathbb C({\bf i})$ and of order induced by $(\alpha,\beta)$ are defined by

\begin{align}%\label{SRFracDer}
& ({_{RL}}D _{a^+,{\bf i} }^{(\alpha,\beta)} f\mid_{{S_{a , b, c, {\bf i} }}})(x+{\bf i}y , u,v) :=
 \nonumber \\
 & ({_{RL}}D _{a^+  }^{ \alpha_0+ {\bf i} \alpha_1}   f\mid_{S_{a , b,c,{\bf i} }}) (x+{\bf i}v) + {\bf i} ({_{RL}}D _{0^+  }^{ \beta_0+ {\bf i} \beta_1}   f\mid_{S_{a , b,c,{\bf i} }})(u+{\bf i}y) \nonumber \end{align}
and
\begin{align*}% \label{SRFracDer1}
&({_{RL}}D _{b^-, {\bf i}}^{(\alpha,\beta)} f\mid_{{S_{a , b, c, {\bf i} }}})(x+{\bf i} y, u,v):=
\nonumber \\
 & ({_{RL}}D _{b^-  }^{\alpha_0+ {\bf i} \alpha_1}   f\mid_{S_{a , b,c,{\bf i} }}) (x+{\bf i}v) + {\bf i} ({_{RL}}D _{c^-  }^{\beta_0+{\bf i} \beta_1}   f\mid_{S_{a , b,c,{\bf i} }})(u+{\bf i}y)
 \nonumber \\
 \end{align*}
respectively.
\end{definition}
\begin{remark}
The slice regular fractional Riemann-Liouville derivatives defined above are quaternionic right-linear operators.
\end{remark}
\begin{remark}
The right versions of the previous differential operators are given by
\begin{align*}  %\label{RightSRFracDer}
& ({_{RL}}D _{a^+,{\bf i},r }^{(\alpha,\beta)} f\mid_{{S_{a , b, c, {\bf i} }}})(x+{\bf i}y , u,v) :=
\nonumber \\
 & ({_{RL}}D _{a^+  }^{\alpha_0+ {\bf i} \alpha_1}   f\mid_{S_{a , b,c,{\bf i} }}) (x+{\bf i}v) +  ({_{RL}}D _{0^+  }^{\beta_0+{\bf i} \beta_1}   f\mid_{S_{a , b,c,{\bf i} }})(u+{\bf i}y) {\bf i} \end{align*}
and
\begin{align*}   % \label{RightSRFracDer1}
&({_{RL}}D _{b^-, {\bf i},r}^{(\alpha, \beta)} f\mid_{{S_{a , b, c, {\bf i} }}})(x+{\bf i} y, u,v):=
\nonumber \\
 & ({_{RL}}D _{b^-  }^{\alpha_0+ {\bf i} \alpha_1}   f\mid_{S_{a , b,c,{\bf i} }}) (x+{\bf i}v) +  ({_{RL}}D _{c^-  }^{\beta_0+ {\bf i} \beta_1}   f\mid_{S_{a , b,c,{\bf i} }})(u+{\bf i}y) {\bf i}.
\end{align*}

\end{remark}

\begin{proposition}\label{FracProp1}
Let $f\in AC^1(S_{a , b,c }, \mathbb H)$. Then
\begin{align*}
& ({_{RL}}D _{a^+,{\bf i} }^{(\alpha,\beta)} f\mid_{{S_{a , b, c, {\bf i} }}})(x+{\bf i}y , u,v)=
\nonumber \\
 & 2\overline{\partial}_{{\bf i}}   \left\{ [{\bf I}_{a^+}^{1-\alpha_0-{\bf i}\alpha_1} f\mid_{S_{a , b,c,{\bf i} }}](x+{\bf i}v) +
[{\bf I}_{0^+}^{1-\beta_0-{\bf i}\beta_1} f\mid_{S_{a , b, c, {\bf i} }}](u+{\bf i}y)\right\}
\end{align*}
and
\begin{align*}
&({_{RL}}D _{b^-, {\bf i}}^{(\alpha,\beta)} f\mid_{{S_{a , b, c, {\bf i} }}})(x+{\bf i} y, u,v)=
\nonumber \\
&(-1)2 \overline{\partial}_{{\bf i}}  \left\{ {\bf I}_{b^-}^{1-\alpha_0-{\bf i}\alpha_1} f\mid_{S_{a , b,c,{\bf i} }}](x+{\bf i}v)   +
  [{\bf I}_{c^-}^{1-\beta_0-{\bf i}\beta_1} f\mid_{S_{a , b, c, {\bf i} }}](u+{\bf i}y)
\right\}.
\end{align*}
 \end{proposition}
\begin{proof}
By using Definition \ref{FDQ} and formulas \eqref{FracDer}, \eqref{FracDer1}, we immediately get
 {
\begin{align}
({_{RL}}D _{a^+,{\bf i} }^{(\alpha,\beta)} f\mid_{{S_{a , b, c, {\bf i} }}})(x+{\bf i}y , u,v)
 &=  \frac{\partial}{\partial x} [{\bf I}_{a^+}^{1-\alpha_0- {\bf i} \alpha_1} f\mid_{S_{a , b,c,{\bf i} }}](x+{\bf i}v) \nonumber \\
 &+ {\bf i} \frac{\partial}{\partial y}
[{\bf I}_{0^+}^{1-\beta_0- {\bf i} \beta_1} f\mid_{S_{a , b, c, {\bf i} }}](u+{\bf i}y)
\end{align}}
and
 {
 \begin{align}
 ({_{RL}}D _{b^-, {\bf i}}^{(\alpha,\beta)} f\mid_{{S_{a , b, c, {\bf i} }}})(x+{\bf i} y, u,v)
 &= -\left( \frac{\partial}{\partial x} [{\bf I}_{b^-}^{1-\alpha_0- {\bf i} \alpha_1} f\mid_{S_{a , b,c,{\bf i} }}](x+{\bf i}v)\right.\nonumber \\
 & +\left. {\bf i} \frac{\partial}{\partial y}
[{\bf I}_{c^-}^{1-\beta_0- {\bf i} \beta_1} f\mid_{S_{a , b, c, {\bf i} }}](u+{\bf i}y)\right)
\end{align}
}
\end{proof}

\begin{definition}(Fractional slice regular function)
Let $f\in AC^1(S_{a,b,c}, \mathbb H)$ such that the mappings
\begin{equation}\label{m1}
x\mapsto [{\bf I}_{a^+}^{1-\alpha_0-{\bf i}\alpha_1} f\mid_{S_{a , b,c,{\bf i} }}](x+{\bf i}v),
\end{equation}
and
\begin{equation}\label{m2}
y\mapsto [{\bf I}_{0^+}^{1-\beta_0-{\bf i}\beta_1} f\mid_{S_{a , b, c, {\bf i} }}](u+{\bf i}y)
\end{equation}
belong to $C^1([a , b], \mathbb H)$  and  $C^1([0 , c], \mathbb H)$, respectively. We will say that $f$ is a fractional slice regular function, in the sense of Riemann-Liouville, of order $ (\alpha ,\beta)$  on $S_{a , b,c }$ if
 \begin{align*}
& ({_{RL}}D _{a^+,{\bf i} }^{(\alpha,\beta)} f\mid_{{S_{a , b, c, {\bf i} }}})(x+{\bf i}y , u,v) =0 ,\quad \textrm{on}
 \  \  {S_{a , b, c, {\bf i} }} ,\end{align*} for all  ${\bf i}\in \mathbb S^2$.
Moreover, if
 \begin{align*}
& ({_{RL}}D _{a^+,{\bf i}, r }^{(\alpha,\beta)} f\mid_{{S_{a , b, c, {\bf i} }}})(x+{\bf i}y , u,v) =0 ,\quad \textrm{on} \  \
 {S_{a , b, c, {\bf i}}},
\end{align*}
for all  ${\bf i}\in \mathbb S^2$ then we call $f$ a right-fractional slice regular function  of order $ (\alpha ,\beta)$ on $S_{a , b,c }$.

We will denote by ${}_{RL}\mathcal{SR}^{(\alpha,\beta)}(S_{a , b, c  })$ the quaternionic right-linear space of all fractional slice regular functions  of order $ (\alpha ,\beta)$ on $S_{a , b,c }$ {and by ${}_{RL}\mathcal{SR}_r^{(\alpha,\beta)}(S_{a , b, c  })$ the quaternionic left-linear space of all right-fractional slice regular functions of order $(\alpha ,\beta)$ on $S_{a , b,c }$}.
\end{definition}

\begin{example} Given $q_1,q_2\in \mathbb H$ and $\delta_0, \delta_1,\gamma_0,\gamma_1\in \mathbb R$ such that $0 < \delta_0,\gamma_0< 1$ let us set
{\small
\begin{align*}
    &  f(x+{\bf i}y)= \\
     & \left\{ 1- {\bf I}_{a^+}^{1-\alpha_0- {\bf i }\alpha_1}  [\frac{(t-a)^{-\alpha_0- {\bf i }\alpha_1}}{\Gamma(1-\alpha_0- {\bf i }\alpha_1)} ](x) \right\}\left\{1- {\bf I}_{0^+}^{1-\beta_0- {\bf i }\beta_1} [\frac{t^{-\beta_0- {\bf i }\beta_1}}{\Gamma(1-\beta_0- {\bf i }\beta_1)}  ](y)\right\} q_1  \\
  & + {\bf i } \left\{    (x-a)^{\delta_0+{\bf i}\delta_1} - {\bf I}_{a^+}^{1-\alpha_0- {\bf i}\alpha_1} [\frac{\Gamma(\delta_0+{\bf i}\delta_1 +1)(x-a)^{\delta_0+{\bf i}\delta_1 -\alpha_0- {\bf i}\alpha_1}}{\Gamma(\delta_0
  + {\bf i}\delta_1 +1-\alpha_0 -{\bf i}\alpha_1)} ](x)\right\} \\
  &  \hspace{.6cm} \left\{   y^{\gamma_0 + {\bf i} \gamma_1} - {\bf I}_{0^+}^{1-\beta_0- {\bf i}\beta_1}
  [\frac{\Gamma(\gamma_0 + {\bf i} \gamma_1 +1)y^{\gamma_0 + {\bf i} \gamma_1-\beta_0 - {\bf i}\beta_1}}
  {\Gamma(\gamma_0 + {\bf i} \gamma_1+1-\beta_0 - {\bf i} \beta_1)} ](y)\right\} q_2,
\end{align*}}

\noindent
for all $x+{\bf i}y \in S_{a , b,c }$.

\noindent
From \eqref{cte} and direct computations it follows that
\begin{align*}
& ({_{RL}}D _{a^+,{\bf i} }^{(\alpha,\beta)} f\mid_{{S_{a , b, c, {\bf i} }}})(x+{\bf i}y , u,v) =0 ,\quad \textrm{on}
 \  \  {S_{a , b, c, {\bf i} }} ,\end{align*} for all ${\bf i}\in \mathbb S^2$ and all
 $(u,v)\in (a,b)\times (0,c)$.
 \end{example}
The relation between the fractional slice regular functions and the classical ones is described in the next result.
\begin{proposition}\label{FracProp2}
Let $f\in AC^1(S_{a , b,c }, \mathbb H)$ such that the mappings (\ref{m1}) and (\ref{m2}) belong to $C^1([a , b], \mathbb H)$ and  $C^1([0 , c], \mathbb H)$, respectively. Then $f\in {}_{RL}\mathcal{SR}^{(\alpha,\beta)}(S_{a , b, c })$ if and only if the mapping
\begin{equation}\label{formula44}
q=x+{\bf i}  y \mapsto  [{\bf I}_{a^+}^{1-\alpha_0-{\bf i}\alpha_1} f\mid_{S_{a , b,c,{\bf i} }}](x+{\bf i}v) +
[{\bf I}_{0^+}^{1-\beta_0+ {\bf i}\beta_1} f\mid_{S_{a , b, c, {\bf i} }}](u+{\bf i}y)
\end{equation}
belongs to $ \mathcal{SR}(S_{a , b, c })$.
\end{proposition}
\begin{proof}
It follows directly from Proposition \ref{FracProp1}.
\end{proof}

\noindent
As a consequence of the Representation Theorem and Splitting Lemma, see Theorem \ref{1.3}, we have the following result:
\begin{proposition}\label{Fract1.3}
Let $f\in {}_{RL}\mathcal{SR}^{(\alpha,\beta)}(S_{a , b, c  }).$ Then
\begin{enumerate}
\item (Splitting lemma in the fractional case)
For every ${\bf i},{\bf j}\in \mathbb{S}$ mutually orthogonal,
there exist $F,G \in AC^{1} (S_{a , b, c ,{\bf i}} , \mathbb C({\bf i}) )$ such that
the mappings
$$x+{\bf i} y \mapsto  [{\bf I}_{a^+}^{1-\alpha_0-{\bf i}\alpha_1} F\mid_{S_{a , b,c,{\bf i} }}](x+{\bf i}v ) +
[{\bf I}_{0^+}^{1-\beta_0 - {\bf i}\beta_1} F\mid_{S_{a , b, c, {\bf i} }}](u+{\bf i}y )  $$
and
$$x+{\bf i} y \mapsto   [{\bf I}_{a^+}^{1-\alpha_0-{\bf i}\alpha_1} G\mid_{S_{a , b,c,{\bf i} }}](x+{\bf i}v) +
[{\bf I}_{0^+}^{1-\beta_0 - {\bf i}\beta_1} G\mid_{S_{a , b, c, {\bf i} }}](u+{\bf i}y )  $$
are holomorphic functions on ${S_{a , b, c, {\bf i} }}$ and $f_{\mid_{\Omega_{\bf i}}} =F +G  {\bf j}$ on ${S_{a , b, c, {\bf i} }}$.

\item (Representation Formula in the fractional case)
For every  $q=x+{\bf i}' y \in S_{a,b,c}$, and every ${\bf i}\in\mathbb S^2$  the following identity holds:
 {\small
\[\begin{split}
[{\bf I}_{a^+}^{1-\alpha_0-{\bf i}'\alpha_1} f\mid_{S_{a , b,c,{\bf i}' }}](x+{\bf i}'v )  & +
[{\bf I}_{0^+}^{1-\beta_0 - {\bf i}\beta_1} f\mid_{S_{a , b, c, {\bf i}' }}](u+{\bf i}'y )  =  \\
& \frac{1}{2}\left\{  {\bf I}_{a^+}^{1-\alpha_0-{\bf i}\alpha_1} [
 f\mid_{S_{a , b,c,{\bf i} }}
 - {\bf i}' {\bf i}  f\mid_{S_{a , b,c,{\bf i} }}  ]  \right.  \\
& \left. +    {\bf I}_{a^+}^{1-\alpha_0+{\bf i}\alpha_1} [
 f\mid_{S_{a , b,c,-{\bf i} }}  +  {\bf i}' {\bf i}
  f\mid_{S_{a , b,c,-{\bf i} }} ] \right\} (x+{\bf i}v )
\\
& +
\frac {1}{2} \left\{ {\bf I}_{0^+}^{1-\beta_0-{\bf i}\beta_1 } [
 f\mid_{S_{a , b,c,{\bf i} }}
    - {\bf i}' {\bf i}
  f\mid_{S_{a , b,c,{\bf i} }}]  \right. \\
 & \left. +
 {\bf I}_{0^+}^{1-\beta_0+{\bf i}\beta_1 }[
 f\mid_{S_{a , b,c,-{\bf i} }} +{\bf i}' {\bf i}
  f\mid_{S_{a , b,c,-{\bf i} }}   ] \right\} (u+{\bf i}y ) .
\end{split} \] }

\noindent
for all ${\bf i}\in \mathbb S^2$.
\end{enumerate}
\end{proposition}
\begin{proof}
We first prove the Splitting Lemma. As the mapping
$$q=x+{\bf i} y \mapsto  [{\bf I}_{a^+}^{1-\alpha_0-{\bf i}\alpha_1} f\mid_{S_{a , b,c,{\bf i} }}](x+{\bf i}v ) +
[{\bf I}_{0^+}^{1-\beta_0 - {\bf i}\beta_1} f\mid_{S_{a , b, c, {\bf i} }}](u+{\bf i}y )$$
belongs to $ \mathcal{SR}(S_{a , b, c })$ we have that:
%\begin{enumerate}
%\item
\[\begin{split}
[{\bf I}_{a^+}^{1-\alpha_0-{\bf i}\alpha_1} f\mid_{S_{a , b,c,{\bf i} }}] (x+{\bf i}v )   +
[{\bf I}_{0^+}^{1-\beta_0 - {\bf i}\beta_1} & f\mid_{S_{a , b, c, {\bf i} }}] (u+{\bf i}y )  = \\
& [{\bf I}_{a^+}^{1-\alpha_0-{\bf i}\alpha_1} f_1\mid_{S_{a , b,c,{\bf i} }}](x+{\bf i}v )  \\
& + [{\bf I}_{0^+}^{1-\beta_0 - {\bf i}\beta_1} f_1\mid_{S_{a , b, c, {\bf i} }}](u+{\bf i}y )  \\
& +\left\{ [{\bf I}_{a^+}^{1-\alpha_0-{\bf i}\alpha_1} f_2\mid_{S_{a , b,c,{\bf i} }}](x+{\bf i}v) \right. \\
&\left. + [{\bf I}_{0^+}^{1-\beta_0 - {\bf i}\beta_1} f_2\mid_{S_{a , b, c, {\bf i} }}](u+{\bf i}y ) \right\}{\bf j},
\end{split}\]

\noindent
where $f\mid{S_{a , b, c ,{\bf i}}} =f_1 + f_2 {\bf j}$ and
$f_1,f_2\in AC^{1}( S_{a , b, c ,{\bf i}} , \mathbb C({\bf i}) )$. Moreover, by the Splitting Lemma for slice regular functions, we have that the mappings
$$x+{\bf i} y \mapsto  [{\bf I}_{a^+}^{1-\alpha_0-{\bf i}\alpha_1} f_1\mid_{S_{a , b,c,{\bf i} }}](x+{\bf i}v ) +
[{\bf I}_{0^+}^{1-\beta_0 - {\bf i}\beta_1} f_1\mid_{S_{a , b, c, {\bf i} }}](u+{\bf i}y )  $$
and
$$x+{\bf i} y \mapsto   [{\bf I}_{a^+}^{1-\alpha_0-{\bf i}\alpha_1} f_2\mid_{S_{a , b,c,{\bf i} }}](x+{\bf i}v) +
[{\bf I}_{0^+}^{1-\beta_0 - {\bf i}\beta_1} f_2\mid_{S_{a , b, c, {\bf i} }}](u+{\bf i}y )$$
are holomorphic functions on ${S_{a , b, c, {\bf i} }}$ and this concludes the proof.

%\item
\noindent
To prove the Representation Formula in the fractional case, we first observe that
the quaternionic conjugation acts on the map \eqref{formula44} as:
\begin{align*}
\bar q=x-{\bf i}  y \mapsto & \  [{\bf I}_{a^+}^{1-\alpha_0+{\bf i}\alpha_1} f\mid_{S_{a , b,c,{\bf i} }}](x-{\bf i}v ) +
[{\bf I}_{0^+}^{1-\beta_0 +{\bf i}\beta_1} f\mid_{S_{a , b, c, {\bf i} }}](u-{\bf i}y) \\
& \ = [{\bf I}_{a^+}^{1-\alpha_0+{\bf i}\alpha_1} f\mid_{S_{a , b,c,-{\bf i} }}](x+{\bf i}v ) +
[{\bf I}_{0^+}^{1-\beta_0+{\bf i}\beta_1} f\mid_{S_{a , b, c, -{\bf i} }}](u+{\bf i}y)  .
\end{align*}

We now consider any $q=x+{\bf i}' y \in S_{a,b,c}$ , and any  ${\bf i}  \in \mathbb S^2$. We have that
\[\begin{split}
[{\bf I}_{a^+}^{1-\alpha_0-{\bf i}'\alpha_1} &  f\mid_{S_{a , b,c,{\bf i}' }}](x +{\bf i}'v ) +
[{\bf I}_{0^+}^{1-\beta_0-{\bf i}'\beta_1 } f\mid_{S_{a , b, c, {\bf i}' }}](u+{\bf i}'y )  = \\
&\hspace{-8mm}  \frac {1}{2}\left\{
[{\bf I}_{a^+}^{1-\alpha_0-{\bf i}\alpha_1} f\mid_{S_{a , b,c,{\bf i} }}](x+{\bf i}v ) + [{\bf I}_{0^+}^{1-\beta_0-{\bf i}\beta_1 } f\mid_{S_{a , b, c, {\bf i} }}](u+{\bf i}y )  \right. \\
&\hspace{-8mm} \left.+
[{\bf I}_{a^+}^{1-\alpha_0-{\bf i}\alpha_1} f\mid_{S_{a , b,c,{\bf i} }}](x-{\bf i}v ) + [{\bf I}_{0^+}^{1-\beta_0-{\bf i}\beta_1 } f\mid_{S_{a , b, c, {\bf i} }}](u-{\bf i}y  )
\right\} \\
&\hspace{-8mm} + \frac {1}{2} {\bf i}' {\bf i}\left\{
[{\bf I}_{a^+}^{1-\alpha_0-{\bf i}\alpha_1} f\mid_{S_{a , b,c,{\bf i} }}](x-{\bf i}v ) + [{\bf I}_{0^+}^{1-\beta_0-{\bf i}\beta_1 } f\mid_{S_{a , b, c, {\bf i} }}](u-{\bf i}y ) \right.\\
&\hspace{-8mm} \left. -
[{\bf I}_{a^+}^{1-\alpha_0-{\bf i}\alpha_1} f\mid_{S_{a , b,c,{\bf i} }}](x+{\bf i}v ) + [{\bf I}_{0^+}^{1-\beta_0-{\bf i}\beta_1 } f\mid_{S_{a , b, c, {\bf i} }}](u+{\bf i}y  )
\right\}\\
&\hspace{-8mm} =
\frac {1}{2}\left\{
[{\bf I}_{a^+}^{1-\alpha_0-{\bf i}\alpha_1} f\mid_{S_{a , b,c,{\bf i} }}](x+{\bf i}v ) +[{\bf I}_{0^+}^{1-\beta_0-{\bf i}\beta_1 } f\mid_{S_{a , b, c, {\bf i} }}](u+{\bf i}y )  \right. \\
&\hspace{-8mm} \left.+
[{\bf I}_{a^+}^{1-\alpha_0+{\bf i}\alpha_1} f\mid_{S_{a , b,c,-{\bf i} }}](x+{\bf i}v ) + [{\bf I}_{0^+}^{1-\beta_0+{\bf i}\beta_1 } f\mid_{S_{a , b, c, -{\bf i} }}](u+{\bf i}y  )
\right\} \\
&\hspace{-8mm} + \frac {1}{2} {\bf i}' {\bf i}\left\{
[{\bf I}_{a^+}^{1-\alpha_0+{\bf i}\alpha_1} f\mid_{S_{a , b,c,-{\bf i} }}](x+{\bf i}v ) + [{\bf I}_{0^+}^{1-\beta_0+{\bf i}\beta_1 } f\mid_{S_{a , b, c, -{\bf i} }}](u+{\bf i}y ) \right.\\
&\hspace{-8mm} \left. -
[{\bf I}_{a^+}^{1-\alpha_0-{\bf i}\alpha_1} f\mid_{S_{a , b,c,{\bf i} }}](x+{\bf i}v )  + [{\bf I}_{0^+}^{1-\beta_0-{\bf i}\beta_1 } f\mid_{S_{a , b, c, {\bf i} }}](u+{\bf i}y  )
\right\}\\
&\hspace{-8mm} =
\frac{1}{2}\left\{  {\bf I}_{a^+}^{1-\alpha_0-{\bf i}\alpha_1} [
 f\mid_{S_{a , b,c,{\bf i} }}
 - {\bf i}' {\bf i}  f\mid_{S_{a , b,c,{\bf i} }}  ]   \right. \\
&\hspace{-8mm} \left. +    {\bf I}_{a^+}^{1-\alpha_0+{\bf i}\alpha_1} [
 f\mid_{S_{a , b,c,-{\bf i} }}  +  {\bf i}' {\bf i}
  f\mid_{S_{a , b,c,-{\bf i} }} ] \right\} (x+{\bf i}v )
\\
&\hspace{-8mm} +
\frac {1}{2} \left\{ {\bf I}_{0^+}^{1-\beta_0-{\bf i}\beta_1 } [
 f\mid_{S_{a , b,c,{\bf i} }}
    - {\bf i}' {\bf i}
  f\mid_{S_{a , b,c,{\bf i} }}] \right. \\
&\hspace{-8mm} \left. + {\bf I}_{0^+}^{1-\beta_0+{\bf i}\beta_1 }[
 f\mid_{S_{a , b,c,-{\bf i} }} +{\bf i}' {\bf i}
  f\mid_{S_{a , b,c,-{\bf i} }}   ] \right\} (u+{\bf i}y ) ,
\end{split}\]

\noindent
for all ${\bf i}\in \mathbb S^2$. The statement follows.
%\end{enumerate}
\end{proof}
The following result is peculiar of the fractional case:
\begin{proposition}\label{Fract1.3.1}
Let $f\in {}_{RL}\mathcal{SR}^{(\alpha,\beta)}(S_{a , b, c  }).$ Then
\[\begin{split}
\frac{y ^{\beta_0 +{\bf i}'\beta_1 -1}}{\Gamma[ \beta_0 +{\bf i}'\beta_1]}  f & \mid_{S_{a , b,c,{\bf i}' }}  (x +{\bf i}'v ) +
   \frac{(x-a)^{\alpha_0+ {\bf i}'\alpha_1-1}}{\Gamma[ \alpha_0+ {\bf i}'\alpha_1]} f\mid_{S_{a , b, c, {\bf i}' }}(u+{\bf i}' y ) =  \\
&
\frac{1}{2}
\frac{v^{-\beta_0 - {\bf i}\beta_1 } } {\Gamma[1-\beta_0 -{\bf i}\beta_1]}
 {_{RL}}D_{0^+}^{1-\alpha_0 -{\bf i}'\alpha_1}
  \left\{   {\bf I}_{a^+}^{1-\alpha_0-{\bf i}\alpha_1} \left[ (
 f
 - {\bf i}' {\bf i}  f)\mid_{S_{a , b,c,{\bf i} }}  \right]   \right. \\
& \left. +    {\bf I}_{a^+}^{1-\alpha_0+{\bf i}\alpha_1} \left[ (
 f   +  {\bf i}' {\bf i}
  f ) \mid_{S_{a , b,c,-{\bf i} }} \right] \right\} (x+{\bf i}v )
\\
 & +
\frac {1}{2} \frac{(u-a)^{-\alpha_0- {\bf i} \alpha_1 }}{\Gamma[1-\alpha_0 - {\bf i}\alpha_1]}  {_{RL}}D_{a^+}^{1-\beta_0 - {\bf i}'\beta_1} \left\{ {\bf I}_{0^+}^{1-\beta_0-{\bf i}\beta_1 } \left[ (
 f      - {\bf i}' {\bf i}
  f)\mid_{S_{a , b,c,{\bf i} }}\right] \right. \\
& \left. +
 {\bf I}_{0^+}^{1-\beta_0+{\bf i}\beta_1 }\left[
( f  +{\bf i}' {\bf i}
  f) \mid_{S_{a , b,c,-{\bf i} }}  \right] \right\} (u+{\bf i}y ) ,
\end{split}\]
%}
\noindent
for all ${\bf i},{\bf i}'\in \mathbb S^2$,  $x   >  a $ and $y > 0$.
\end{proposition}
\begin{proof}
The assertion follows by combining the action of ${_{RL}}D_{0^+}^{1-\alpha_0 -{\bf i}'\alpha_1}  \circ {_{RL}}D_{a^+}^{1-\beta_0 - {\bf i}'\beta_1} $ with the Fundamental Theorem for Riemann-Liouville fractional calculus, see \eqref{FundTheorem}, and with \eqref{cte}.
\end{proof}
In particular, we have:
\begin{corollary}
Let $\alpha_1=\beta_1=0$. Then for any  $f\in {}_{RL}\mathcal{SR}^{(\alpha,\beta)}(S_{a , b, c  })$ we have that

\begin{align*}
&  \frac{y ^{\beta_0   -1}}{\Gamma[ \beta_0  ]}  f\mid_{S_{a , b,c,{\bf i}' }} (x+{\bf i}'v ) +
   \frac{(x-a)^{\alpha_0 -1}}{\Gamma[ \alpha_0 ]} f\mid_{S_{a , b, c, {\bf i}' }}(u+{\bf i}'y )  \\
=&
\frac{1}{2} \left\{
\frac{v^{-\beta_0 } } {\Gamma[1-\beta_0  ]}
 (  f
 - {\bf i}' {\bf i}  f)\mid_{S_{a , b,c,{\bf i} }}   (x+ {\bf i} v)
+ \frac{v^{-\beta_0 } } {\Gamma[1-\beta_0  ]}
   (
 f   +  {\bf i}' {\bf i}
  f ) \mid_{S_{a , b,c,-{\bf i} }}   (x+{\bf i}v )  \right.
\\
&
\left.
 +
 \frac{(u-a)^{-\alpha_0}}{\Gamma[1-\alpha_0 ]}   (
 f      - {\bf i}' {\bf i}
  f)\mid_{S_{a , b,c,{\bf i} }}   (u+{\bf i}y )
 +
\frac{(u-a)^{-\alpha_0}}{\Gamma[1-\alpha_0 ]}   (
 f  +{\bf i}' {\bf i}
  f) \mid_{S_{a , b,c,-{\bf i} }}  ) (u+{\bf i} y)  \right\} ,
\end{align*}

\noindent
for all $x   >  a $ and $y > 0$.
\end{corollary}
In the next result we describe a series expansion associated with a fractional Riemann-Liouville slice regular function:
\begin{proposition}\label{Fract1.4.1}
Suppose there exists $r> 0$ such that $a+r < b $ and let $f\in {}_{RL}\mathcal{SR}^{(\alpha, \beta)}(S_{a-r , b, c  }).$
%\begin{enumerate}
%\item
If
  $$ \sum_{n=1}^{\infty}  \sum_{k=1}^{n-1} \left(  \begin{array}{c} n \\ k \end{array}  \right)
   \frac{\Gamma (n-k+1)\Gamma (k+1)}{\Gamma( n-k+\alpha_0)\Gamma(k+\alpha_1)}
   (x-a)^{n-k }
   y^{k} {\bf i}^k a_n , $$
   $$  \sum_{n=1}^{\infty}
 \frac{\Gamma (n+1)}{\Gamma(n+\alpha_0)} (x-a)^{n}
  a_n,  \ \  \textrm{and}  \  \  \  \sum_{n=1}^{\infty}
   \frac{\Gamma (n+1)}{\Gamma(n+\alpha_1)}y^{n} {\bf i}^n a_n$$

\noindent
are uniformly convergent series on $\mathbb B^4(a, r)$, then  we have
  \begin{align*}
&  \frac{
   (x-a)^{1-\alpha_0-{\bf i}\alpha_1} }{\Gamma( \beta_0 +{\bf i}\beta_1)}  f\mid_{S_{a-r , b,c,{\bf i} }} (x+{\bf i}v) +
   \frac{y^{1-\beta_0 -{\bf i}\beta_1}
   }{\Gamma( \alpha_0+{\bf i}\alpha_1)} f\mid_{S_{a-r , b, c, {\bf i} }}(u+{\bf i}y )  \\
 =&  \sum_{n=0}^{\infty}  \sum_{k=0}^{n}
  \lambda_{k,n,\alpha,\beta}  \, (x-a)^{n-k}
    y^{k-1} {\bf i}^k a_n ,
\end{align*}
	
\noindent
for any $x+{\bf i}y\in \mathbb B^4(a, r)$, where
\begin{align*}
 \lambda_{k,n,\alpha,\beta}:= & \left(  \begin{array}{c} n \\ k \end{array}  \right)
  \frac{\Gamma (n-k+1)}{\Gamma(n-k+\alpha_0+{\bf i}\alpha_1)}  \frac{\Gamma (k+1)}{\Gamma(k+\beta_0 + {\bf i}\beta_1)}\\
  =&  \frac{\Gamma(n+1)} {\Gamma(n-k+\alpha_0+{\bf i}\alpha_1)  \Gamma(k+\beta_0 + {\bf i}\beta_1)} .
  \end{align*}
%\item

%\end{enumerate}
\end{proposition}
\begin{proof}
%\begin{enumerate}
%\item

\noindent
The mapping
$$q=x+{\bf i} y \mapsto  [{\bf I}_{a^+}^{1-\alpha_0- {\bf i}\alpha_1} f\mid_{S_{a-r , b,c,{\bf i} }}](x+{\bf i}v ) +
[{\bf I}_{0^+}^{1-\beta_0- {\bf i}\beta_1} f\mid_{S_{a-r , b, c, {\bf i} }}](u+{\bf i}y)$$
belongs to $ \mathcal{SR}^{\alpha}(S_{a-r , b, c })$ and its expansion in power series in a neighborhood of the point $a$ is
\begin{align*}
& [{\bf I}_{a^+}^{1-\alpha_0-{\bf i}\alpha_1} f\mid_{S_{a-r , b,c,{\bf i} }}](x+{\bf i}v ) +
[{\bf I}_{0^+}^{1-\beta_0 - {\bf i}\beta_1} f\mid_{S_{a -r , b, c, {\bf i} }}](u+{\bf i}y ) \\
= &  \sum_{n=0}^{\infty}  (x-a+{\bf i} y)^n a_n =   \sum_{n=0}^{\infty}  \sum_{k=0}^n \left(  \begin{array}{c} n \\ k \end{array}  \right) (x-a)^{n-k}  y^k{\bf i}^k a_n \\
=  & a_0  +  \sum_{n=1}^{\infty}  \sum_{k=1}^{n-1} \left(  \begin{array}{c} n \\ k \end{array}  \right) (x-a)^{n-k}  y^k{\bf i}^k a_n
 +  \sum_{n=1}^{\infty}     (x-a)^{n-k}    a_n
  +  \sum_{n=1}^{\infty}     y^n{\bf i}^n a_n ,
\end{align*}
	
\noindent
where $(a_n) $ is a sequence of quaternions.

\noindent
After the action of ${_{RL}}D_{0^+}^{1-\beta_0 - {\bf i}\beta_1}  \circ {_{RL}}D_{a^+}^{1-\alpha_0- {\bf i}\alpha_1}$ and \eqref{cte} we obtain that
 \begin{align*}
&  \frac{y^{\beta_0 +{\bf i}\beta_1-1}}{\Gamma( \beta_0 +{\bf i}\beta_1)}  f\mid_{S_{a , b,c,{\bf i} }} (x+{\bf i}v  ) +
   \frac{(x-a)^{\alpha_0 +{\bf i}\alpha_1-1}}{\Gamma( \alpha_0 +{\bf i}\alpha_1 )} f\mid_{S_{a , b, c, {\bf i} }}(u+{\bf i}y  )  \\
=
 & a_0
 {_{RL}}D_{0^+}^{1-\beta_0 -{\bf i}\beta_1}[1]   {_{RL}}D_{a^+}^{1-\alpha_0- {\bf i}\alpha_1}[1]
  \\
  & +      \sum_{n=1}^{\infty}  \sum_{k=1}^{n-1} \left(  \begin{array}{c} n \\ k \end{array}  \right)  {_{RL}}D_{a^+}^{1-\alpha_0-{\bf i}\alpha_1}  [(x-a)^{n-k}]
    {_{RL}}D_{0^+}^{1-\beta_0 -{\bf i}\beta_1}[    y^k] {\bf i}^k a_n  \\
& +  \sum_{n=1}^{\infty}    {_{RL}}D_{a^+}^{1-\alpha_0-{\bf i}\alpha_1}  [(x-a)^{n}]  a_n
  +  \sum_{n=1}^{\infty}    {_{RL}}D_{0^+}^{1-\beta_0 -{\bf i}\beta_1}[    y^n] {\bf i}^n a_n ,
\end{align*}

\noindent
for $x   >  a $ and $y > 0$, where we have used the uniform convergence in a neighborhood of $a$.

\noindent
Using formula \eqref{cte} in \cite{VTRMB} we obtain that

\[\begin{split}
\frac{y^{\beta_0 +{\bf i}\beta_1-1}}{\Gamma( \beta_0 +{\bf i}\beta_1)} &  f\mid_{S_{a-r , b,c,{\bf i} }} (x+{\bf i}v) +
   \frac{(x-a)^{\alpha_0+{\bf i}\alpha_1-1}}{\Gamma( \alpha_0+{\bf i}\alpha_1)} f\mid_{S_{a-r , b, c, {\bf i} }}(u+{\bf i}y' )  = \\
 & a_0  \frac{(x-a)^{ \alpha_0+ {\bf i}\alpha_1 -1}}{\Gamma(\alpha_0+{\bf i}\alpha_1)}  \frac{y^{ \beta_0 +{\bf i}\beta_1-1}}{\Gamma(\beta_0 +{\bf i}\beta_1)}
   \\
   &   +      \sum_{n=1}^{\infty}  \sum_{k=1}^{n-1}
   \left(  \begin{array}{c} n \\ k \end{array}  \right)
   \frac{\Gamma (n-k+1)}{\Gamma(n-k+\alpha_0+{\bf i}\alpha_1)} (x-a)^{n-k-1+ \alpha_0+{\bf i}\alpha_1} \\
    & \times \frac{\Gamma (k+1)}{\Gamma(k+\beta_0+{\bf i}\beta_1)} y^{k-1+ \beta_0 +{\bf i}\beta_1} {\bf i}^k a_n  \\
    &
 + \frac{y^{ \beta_0 +{\bf i}\beta_1-1}}{\Gamma(\beta_0 +{\bf i}\beta_1)} \sum_{n=1}^{\infty}
 \frac{\Gamma (n+1)}{\Gamma(n+\alpha_0+{\bf i}\alpha_1)} (x-a)^{n-1+ \alpha_0+{\bf i}\alpha_1}
  a_n \\
  &    +
  \frac{(x-a)^{ \alpha_0+{\bf i}\alpha_1-1}}{\Gamma(\alpha_0+{\bf i}\alpha_1)}
   \sum_{n=1}^{\infty}
   \frac{\Gamma (n+1)}{\Gamma(n+\beta_0+{\bf i}\beta_1)} y^{n-1+ \beta_0 +{\bf i}\beta_1} {\bf i}^n a_n .
\end{split}\]

\noindent
Therefore
\[\begin{split}
\frac{y^{\beta_0 +{\bf i}\beta_1-1}}{\Gamma( \beta_0 +{\bf i}\beta_1)} &  f\mid_{S_{a-r , b,c,{\bf i} }} (x+{\bf i}v) +
   \frac{(x-a)^{\alpha_0+{\bf i}\alpha_1-1}}{\Gamma( \alpha_0+{\bf i}\alpha_1)} f\mid_{S_{a-r , b, c, {\bf i} }}(u+{\bf i}y )  = \\
 & a_0  \frac{(x-a)^{ \alpha_0+ {\bf i}\alpha_1 -1}}{\Gamma(\alpha_0+{\bf i}\alpha_1)}  \frac{y^{ \beta_0 +{\bf i}\beta_1-1}}{\Gamma(\beta_0 +{\bf i}\beta_1)}
   \\
   &   +      \sum_{n=1}^{\infty}  \sum_{k=1}^{n-1}
   \left(  \begin{array}{c} n \\ k \end{array}  \right)
   \frac{\Gamma (n-k+1)}{\Gamma(n-k+\alpha_0+{\bf i}\alpha_1)}  \frac{\Gamma (k+1)}{\Gamma(k+\beta_0 +{\bf i} \beta_1)} \\
   & \times  (x-a)^{n-k-1+ \alpha_0+{\bf i}\alpha_1}
    y^{k-1+ \beta_0 +{\bf i}\beta_1} {\bf i}^k a_n  \\
    &
 + \frac{y^{ \beta_0 +{\bf i}\beta_1-1}}{\Gamma(\beta_0 +{\bf i}\beta_1)} \sum_{n=1}^{\infty}
 \frac{\Gamma (n+1)}{\Gamma(n+\alpha_0+{\bf i}\alpha_1)} (x-a)^{n-1+ \alpha_0+{\bf i}\alpha_1}
  a_n \\
  &    +
  \frac{(x-a)^{ \alpha_0+{\bf i}\alpha_1-1}}{\Gamma(\alpha_0+{\bf i}\alpha_1)}
   \sum_{n=1}^{\infty}
   \frac{\Gamma (n+1)}{\Gamma(n+\beta_0+{\bf i}\beta_1)} y^{n-1+ \beta_0 +{\bf i}\beta_1} {\bf i}^n a_n \\
   & =
       \sum_{n=0}^{\infty}  \sum_{k=0}^{n}
   \left(  \begin{array}{c} n \\ k \end{array}  \right)
   \frac{\Gamma (n-k+1)}{\Gamma(n-k+\alpha_0+ {\bf i}\alpha_1)}  \frac{\Gamma (k+1)}{\Gamma(k+\beta_0 +{\bf i}\beta_1)} \\
   & \times (x-a)^{n-k-1+ \alpha_0+{\bf i}\alpha_1}
    y^{k-1+ \beta_0 +{\bf i}\beta_1} {\bf i}^k a_n .
\end{split}\]
%\item
\end{proof}
\begin{proposition}
\noindent
Let $s > 0$ and  let $${\Delta_q(a,s)}:=
  \{ x + y{\bf i} \  \mid \  (x -a)^2 + y^2 < s^2\}
 \subset  S_{a-r , b,c }.$$ Then, for $q=x+\mathbf{i}_q y\in {\Delta_q(a,s)}$ we have
 {\small
    \begin{align*}
&  \frac{ y^{\beta_0 +{\bf i}\beta_1-1}}{\Gamma( \beta_0 +{\bf i}\beta_1)}  f\mid_{S_{a , b,c,{\bf i} }} (x+{\bf i}v  ) +
   \frac{(x-a)^{\alpha_0+{\bf i}\alpha_1-1}}{\Gamma( \alpha_0+{\bf i}\alpha_1)} f\mid_{S_{a , b, c, {\bf i} }}(u+{\bf i}y  )  \\
=
&
\frac{1}{2\pi }\int_{\partial\Delta_q(a,r) }  \mathcal N_{a}^{(\alpha,\beta)}(\zeta, q ) d\zeta_{{\bf i}_q}
\left\{  [{\bf I}_{a^+}^{1-\alpha_0-{\bf i}\alpha_1} f\mid_{S_{a-r , b,c,{\bf i} }}](\Re \zeta+{\bf i}v ) \right.
\\
& \hspace{2cm}
\left.+
[{\bf I}_{a^+}^{1-\beta_0 -{\bf i}\beta_1} f\mid_{S_{a-r , b, c, {\bf i} }}](u+\Im \zeta \ {\bf i} )  \right\},
\end{align*}
}

\noindent
where $$ \mathcal N_{a}^{(\alpha,\beta)}(\zeta, q ) = \sum_{n=0}^{\infty}\left[  {_{RL}}D_{0^+}^{1-\beta_0 -{\bf i}\beta_1}  \circ {_{RL}}D_{a^+}^{1-\alpha_0-{\bf i}\alpha_1}(q-a)^{   n} \right]
  (\zeta-a)^{-(n+1)}.$$		
\end{proposition}
\begin{proof}

\noindent
Let $q=x+{\bf i}_qy\in \overline{\Delta_q(a,s)} \subset  S_{a-r , b,c }.$ Then
{\small
\begin{align*}
& [{\bf I}_{a^+}^{1-\alpha_0-{\bf i}\alpha_1} f\mid_{S_{a-r , b,c,{\bf i}_q }}](x+{\bf i}_q v ) +
[{\bf I}_{0^+}^{1-\beta_0 -{\bf i}\beta_1} f\mid_{S_{a-r , b, c, {\bf i}_q }}](u+{\bf i}_q y ) \\
=
&
\frac{1}{2\pi }\int_{\partial\Delta_q(a,r) } S^{-1}(\zeta,q) d\zeta_{{\bf i}_q}
\left\{  [{\bf I}_{a^+}^{1-\alpha_0-{\bf i}\alpha_1} f\mid_{S_{a-r , b,c,{\bf i} }}](\Re \zeta+{\bf i}v ) \right.  \\
& \hspace{2cm} \left.+
[{\bf I}_{0^+}^{1-\beta_0 - {\bf i}\beta_1} f\mid_{S_{a-r , b, c, {\bf i} }}](u+\Im \zeta \ {\bf i} )  \right\},
\end{align*}
}
%\end{enumerate}

\noindent
 {where $\zeta=u+{\bf i } v$.}

\noindent
Recall that $a\in \mathbb R$ and $q$ is the quaternionic variable in which the slice regularity is expressed. Due to the uniform convergence of
\begin{align*}
 S^{-1}(\zeta -a,q-a) = \sum_{n=0}^{\infty} (q-a)^{ n}\ (\zeta-a)^{-(n+1)},\quad \forall \zeta\in  \partial\Delta_q(a,r)  \cap \mathbb C({\mathbf i}_q),
\end{align*}

\noindent
we have that
{\small
\begin{align*}
& [{\bf I}_{a^+}^{1-\alpha_0-{\bf i}\alpha_1} f\mid_{S_{a-r , b,c,{\bf i}_q }}](x+{\bf i}_q v ) +
[{\bf I}_{0^+}^{1-\beta_0 -{\bf i}\beta_1} f\mid_{S_{a-r , b, c, {\bf i}_q }}](u+{\bf i}_q y ) \\
=
&
\sum_{n=0}^{\infty} \frac{1}{2\pi }\int_{\partial\Delta_q(a,r) }
(q-a)^{   n}   (\zeta-a)^{-(n+1)}  d\zeta_{{\bf i}_q}
\left\{  [{\bf I}_{a^+}^{1-\alpha_0-{\bf i}\alpha_1 } f\mid_{S_{a-r , b,c,{\bf i} }}](\Re \zeta+{\bf i}v ) \right.  \\
& \hspace{2cm} \left.+
[{\bf I}_{0^+}^{1-\beta_0-{\bf i}\beta_1} f\mid_{S_{a-r , b, c, {\bf i} }}](u+\Im \zeta \ {\bf i} )  \right\}\\
=
&
\sum_{n=0}^{\infty} (q-a)^{   n}  \frac{1}{2\pi }\int_{\partial\Delta_q(a,r) }
  (\zeta-a)^{- (n+1)}  d\zeta_{{\bf i}_q}
\left\{  [{\bf I}_{a^+}^{1-\alpha_0-{\bf i}\alpha_1} f\mid_{S_{a-r , b,c,{\bf i} }}](\Re \zeta+{\bf i}v ) \right.  \\
& \hspace{2cm} \left.+
[{\bf I}_{0^+}^{1-\beta_0-{\bf i}\beta_1} f\mid_{S_{a-r , b, c, {\bf i} }}](u+\Im \zeta \ {\bf i} )  \right\}.
\end{align*}
}
%}
%\end{enumerate}

\noindent
Acting on  both sides  ${_{RL}}D_{0^+}^{1-\alpha_1}  \circ {_{RL}}D_{a^+}^{1-\alpha_0}$  we get that
{\small
\begin{align*}
&  \frac{y^{\beta_0+{\bf i}\beta_1-1}}{\Gamma( \beta_0+{\bf i}\beta_1)}  f\mid_{S_{a , b,c,{\bf i} }} (x+{\bf i}v  ) +
   \frac{(x-a)^{\alpha_0+{\bf i}\alpha_1-1}}{\Gamma( \alpha_0+{\bf i}\alpha_1 )} f\mid_{S_{a , b, c, {\bf i} }}(u+{\bf i}y  )  \\
=
&
\sum_{n=0}^{\infty} \left[ {_{RL}}D_{0^+}^{1-\beta_0 -{\bf i}\beta_1}  \circ {_{RL}}D_{a^+}^{1-\alpha_0-{\bf i}\alpha_1}(q-a)^{   n}  \right]   \frac{1}{2\pi }\int_{\partial\Delta_q(a,r) }
  (\zeta-a)^{-(n+1)}  d\zeta_{{\bf i}_q} \\
& \left\{  [{\bf I}_{a^+}^{1-\alpha_0-{\bf i}\alpha_1} f\mid_{S_{a-r , b,c,{\bf i} }}](\Re \zeta+{\bf i}v )  +
[{\bf I}_{0^+}^{1-\beta_0 - {\bf i}\beta_1} f\mid_{S_{a-r , b, c, {\bf i} }}](u+\Im \zeta \ {\bf i} )  \right\}
\\
=
& \frac{1}{2\pi }\int_{\partial\Delta_q(a,r) }
\sum_{n=0}^{\infty}\left[  {_{RL}}D_{0^+}^{1-\beta_0 -{\bf i}\beta_1}  \circ {_{RL}}D_{a^+}^{1-\alpha_0-{\bf i}\alpha_1}(q-a)^{   n} \right]
  (\zeta-a)^{-(n+1)}  d\zeta_{{\bf i}_q} \\
& \left\{  [{\bf I}_{a^+}^{1-\alpha_0-{\bf i}\alpha_1} f\mid_{S_{a-r , b,c,{\bf i} }}](\Re \zeta+{\bf i}v )  +
[{\bf I}_{0^+}^{1-\beta_0 -{\bf i}\beta_1} f\mid_{S_{a-r , b, c, {\bf i} }}](u+\Im \zeta \ {\bf i} )  \right\}.
\end{align*}
}
%}
%\end{enumerate}

\end{proof}

\begin{remark} Given $n\in \mathbb N \cup \{0\}$. By direct computations we obtain that
\[\begin{split}
{_{RL}}D_{0^+}^{1-\beta_0+{\bf i} \beta_1} &   \circ {_{RL}}D_{a^+}^{1-\alpha_0-{\bf i}\alpha_1}(q-a)^{   n}  = \\
&
     \sum_{k=0}^{n}\left( \begin{array}{c}   n \\ k \end{array} \right)
              \frac{ \Gamma (k+1) (x-a)^{k+ \alpha_0 +{\bf i}\alpha_1 -1} }{ \Gamma(k+\alpha_0) } \\
              & \times  \frac{ \Gamma (n-k+1) y^{n-k+ \beta_0+{\bf i}\beta_1 -1}{\bf i}^{n-k} }{ \Gamma(n-k+\beta_0 +{\bf i}\beta_1) }\\
              & =
     \sum_{k=0}^{n}\left( \begin{array}{c}   n \\ k \end{array} \right)
              \frac{ \Gamma (k+1)  \Gamma (n-k+1)  }{ \Gamma(k+\alpha_0+{\bf i}\alpha_1) \Gamma(n-k+\beta_0 +{\bf i}\beta_1)}  \\
              & \times  (x-a)^{k+ \alpha_0 +{\bf i}\alpha_1-1} y^{n-k+ \beta_0+{\bf i}\beta_1 -1}{\bf i}^{n-k}  \\
              & =
  (x-a)^{ \alpha_0+{\bf i}\alpha_1 -1} y^{ \beta_0 +{\bf i}\beta_1 -1}    \\
  & \times  \sum_{k=0}^{n}\left( \begin{array}{c}   n \\ k \end{array} \right)
              \frac{ \Gamma (k+1)  \Gamma (n-k+1)  }{ \Gamma(k+\alpha_0+{\bf i} \alpha_1) \Gamma(n-k+\beta_0 +{\bf i}\beta_1)} \\
  & \times  (x-a)^{k} y^{n-k} {\bf i}^{n-k}.
\end{split}\]

\noindent
Therefore,
\[\begin{split}
\mathcal N_{a}^{(\alpha,\beta)}(\zeta, q ) & = (x-a)^{ \alpha_0 +{\bf i}\alpha_1-1} y^{ \beta_0 +{\bf i}\beta_1 -1}  \sum_{n=0}^{\infty}       \sum_{k=0}^{n} \\
& \times   \frac{\Gamma(n+1)  }{ \Gamma(k+\alpha_0+{\bf i}\alpha_1) \Gamma(n-k+\beta_0 +{\bf i}\beta_1)}     (x-a)^{k} y^{n-k} {\bf i}^{n-k}.
\end{split} \]
\end{remark}
\section{Fractional slice regular functions in the Caputo sense}

With the notations introduced in the previous section, and partially relying on the same type of reasonings, in this section we study fractional slice regular functions in the Caputo sense.
\begin{definition}
The fractional derivatives in the Caputo sense on the left and on the right associated to the slice $\mathbb C({\bf i})$    {and with  order induced by $(\alpha,\beta)$} are given by:

\begin{align}\label{SRFracDerCaputo}
& ({_{{\mathcal C}}}D _{a^+,{\bf i} }^{(\alpha,\beta)} f\mid_{{S_{a , b, c, {\bf i} }}})(x+{\bf i}y , u,v):=
\nonumber \\
 &  [{\bf I}_{a^+}^{1-\alpha_0-{\bf i}\alpha_1}\frac{\partial}{\partial x} f\mid_{S_{a , b,c,{\bf i} }}](x+{\bf i}v) + {\bf i}
[{\bf I}_{0^+}^{1-\beta_0-{\bf i}\beta_1 } \frac{\partial}{\partial y}f\mid_{S_{a , b, c, {\bf i} }}](u+{\bf i}y) \end{align}
and
\begin{align} \label{SRFracDer1Caputo}
&({_{{\mathcal C}}}D _{b^-, {\bf i}}^{(\alpha,\beta)} f\mid_{{S_{a , b, c, {\bf i} }}})(x+{\bf i} y, u,v):=
\nonumber \\
& (-1)\left\{ [{\bf I}_{b^-}^{1-\alpha_0-{\bf i}\alpha_1} \frac{\partial}{\partial x}f\mid_{S_{a , b,c,{\bf i} }}](x+{\bf i}v)
 + {\bf i}  [{\bf I}_{c^-}^{1-\beta_0-{\bf i}\beta_1} \frac{\partial}{\partial y} f\mid_{S_{a , b, c, {\bf i} }}](u+{\bf i}y)
\right\},
\end{align}
respectively.
\end{definition}
The right version of the previous fractional derivatives  associated to the slice $\mathbb C({\bf i})$ are given by
\begin{align*}
& ({_{{\mathcal C}}}D _{a^+,{\bf i},r}^{(\alpha,\beta)} f\mid_{{S_{a , b, c, {\bf i} }}})(x+{\bf i}y , u,v):=
\nonumber \\
 &  [{\bf I}_{a^+}^{1-\alpha_0-{\bf i}\alpha_1}\frac{\partial}{\partial x} f\mid_{S_{a , b,c,{\bf i} }}](x+{\bf i}v) +
[{\bf I}_{0^+}^{1-\beta_0-{\bf i}\beta_1 } \frac{\partial}{\partial y}f\mid_{S_{a , b, c, {\bf i} }}](u+{\bf i}y)  {\bf i} \end{align*}
and
\begin{align*}
&({_{{\mathcal C}}}D _{b^-, {\bf i},r}^{(\alpha,\beta)} f\mid_{{S_{a , b, c, {\bf i} }}})(x+{\bf i} y, u,v):=
\nonumber \\
& (-1)\left\{ [{\bf I}_{b^-}^{1-\alpha_0-{\bf i}\alpha_1} \frac{\partial}{\partial x}f\mid_{S_{a , b,c,{\bf i} }}](x+{\bf i}v)
 +   [{\bf I}_{c^-}^{1-\beta_0-{\bf i}\beta_1} \frac{\partial}{\partial y} f\mid_{S_{a , b, c, {\bf i} }}](u+{\bf i}y) {\bf i}
\right\} .
\end{align*}
\begin{definition}
Let $f\in AC^1(S_{a , b,c }, \mathbb H)$ such that the mappings (\ref{m1}) and (\ref{m2}) belong to $C^1([a , b], \mathbb H)$  and  $C^1([0 , c], \mathbb H)$, respectively. Then  $f$ is called a fractional slice regular function in the Caputo sense on $S_{a , b,c }$ of order $(\alpha, \beta)$ if
 \begin{align*}
 & ({_{{\mathcal C}}}D _{a^+,{\bf i} }^{(\alpha,\beta)} f\mid_{{S_{a , b, c, {\bf i} }}})(x+{\bf i}y , u,v) =0 ,\quad \textrm{on}
 \  \  {S_{a , b, c, {\bf i} }},
\end{align*}
for all  ${\bf i}\in \mathbb S^2$.

If
 \begin{align*}
& ({_{{\mathcal C}}}D _{a^+,{\bf i},r }^{(\alpha,\beta)} f\mid_{{S_{a , b, c, {\bf i} }}})(x+{\bf i}y , u,v) =0 ,\quad \textrm{on} \  \
 {S_{a , b, c, {\bf i} }},
\end{align*}
for all ${\bf i}\in \mathbb S^2$ we shall call $f$ a right-fractional slice regular function of order $(\alpha, \beta)$ on $S_{a , b,c }$ in the Caputo sense.

The quaternionic right-linear space of fractional slice regular functions, in the sense of Caputo of order $(\alpha, \beta)$ on $S_{a , b,c }$, will be denoted by ${}_{{\mathcal C}}\mathcal{SR}^{(\alpha,\beta)}(S_{a , b, c  })$.  {Furthermore, we write ${}_{{\mathcal C}}\mathcal{SR}_r^{(\alpha,\beta)}(S_{a , b, c  })$ for the quaternionic left-linear space of right-fractional slice regular functions in the sense of Caputo of order $(\alpha, \beta)$ on $S_{a , b,c }$.}
 \end{definition}

\begin{proposition}\label{prop53}
Let $f\in AC^1(S_{a , b,c }, \mathbb H)$ such that the mappings (\ref{m1}) and (\ref{m2}) belong to $C^1([a , b], \mathbb H)$ and $C^1([0 , c], \mathbb H)$, respectively. Then
\[\begin{split}
({_{{\mathcal C}}}D _{a^+,{\bf i} }^{(\alpha,\beta)} & f\mid_{{S_{a , b, c, {\bf i} }}})(x+{\bf i}y , u,v)  \nonumber = \\
& ({_{RL}}D _{a^+,{\bf i} }^{(\alpha, \beta)} f\mid_{{S_{a , b, c, {\bf i} }}})(x+{\bf i}y , u,v)   \\
& - \frac{  (x-a)^{\alpha_0+{\bf i}\alpha_1}}{\Gamma(1-\alpha_0-{\bf i}\alpha_1)}   f(a+ {\bf i} v)  \\
& - {\bf i} \frac{  y^{\beta_0 +{\bf i}\beta_1}}{\Gamma(1-\beta_0 -{\bf i} \beta_1)}  f(u)
\end{split}\]
\end{proposition}
\begin{proof} Formula \eqref{RLandC} shows that
\[\begin{split}
({_{{\mathcal C}}}D _{a^+,{\bf i} }^{(\alpha,\beta)} & f\mid_{{S_{a , b, c, {\bf i} }}})(x+{\bf i}y , u,v)  \nonumber  = \\
&\hspace{-8mm}   {_{RL}} D _{a^+}^{\alpha_0+{\bf i}\alpha_1} ( \  f (t+{\bf i} v) - f(a+ {\bf i} v) \ ) (x) + {\bf i} {_{RL}} D _{0^+}^{\beta_0 +{\bf i}\beta_1} (\ f (u+{\bf i} t) - f(u) \  ) (y) \\
&\hspace{-8mm} =  ({_{RL}}D _{a^+,{\bf i} }^{(\alpha ,\beta)} f\mid_{{S_{a , b, c, {\bf i} }}})(x+{\bf i}y , u,v)  - {_{RL}} D _{a^+}^{\alpha_0 +{\bf i}\alpha_1} (  f(a+ {\bf i} v) ) (x) \\
&\hspace{-8mm}  -  {\bf i}  {_{RL}} D _{0^+}^{\beta_0 +{\bf i}\beta_1} (f(u) ) (y) \\
&\hspace{-8mm} = ({_{RL}}D _{a^+,{\bf i} }^{(\alpha,\beta)} f\mid_{{S_{a , b, c, {\bf i} }}})(x+{\bf i}y , u,v) - \frac{  (x-a)^{\alpha_0+{\bf i}\alpha_1}}{\Gamma(1-\alpha_0-{\bf i}\alpha_1)}   f(a+ {\bf i} v)   \\
&\hspace{-8mm}  - {\bf i} \frac{  y^{\beta_0 +{\bf i}\beta_1}}{\Gamma(1-\beta_0-{\bf i}\beta_1)}  f(u)
\end{split}\]
and the statement follows.
\end{proof}

\begin{corollary}
Under the hypothesis of Proposition \ref{prop53}, we have that $f\in {}_{{\mathcal C}}\mathcal{SR}^{(\alpha, \beta)}(S_{a , b, c  })$ if and only if
\[\begin{split}({_{RL}}D _{a^+,{\bf i} }^{(\alpha,\beta)} & f\mid_{{S_{a , b, c, {\bf i} }}})(x+{\bf i}y , u,v) = \\
& \frac{  (x-a)^{\alpha_0+{\bf i}\alpha_1}}{\Gamma(1-\alpha_0-{\bf i}\alpha_1)}f(a+ {\bf i} v)  + {\bf i}  \frac{  y^{\beta_0 +{\bf i}\beta_1}}{\Gamma(1-\beta_0-{\bf i}\beta_1)}f(u).
\end{split}\]
In addition, for $f\in AC^1(S_{a , b,c }, \mathbb H)$ such that the mappings (\ref{m1}) and (\ref{m2}) belong to $C^1([a , b], \mathbb H)$ and  $C^1([0 , c], \mathbb H)$, respectively, and also $f(a+ {\bf i} v)  = f(u) =0$ for all ${\bf i}\in \mathbb S^2$ we have that $f\in {}_{{\mathcal C}}\mathcal{SR}^{(\alpha,\beta)}(S_{a , b, c  })$ if and only if $f\in   {}_{{RL}}\mathcal{SR}^{(\alpha,\beta)}(S_{a , b, c})$.
\end{corollary}

\section*{Statements and Declarations}
\subsection*{Funding} This work was partially supported by Instituto Polit\'ecnico Nacional (grant numbers SIP20232103, SIP20230312) and CONACYT.
\subsection*{Competing Interests} The authors declare that they have no competing interests regarding the publication of this paper.
\subsection*{Author contributions} All authors contributed equally to the study, read and approved the final version of the submitted manuscript.
\subsection*{Availability of data and material} Not applicable
\subsection*{Code availability} Not applicable
\subsection*{ORCID}
\noindent
Jos\'e Oscar Gonz\'alez-Cervantes: https://orcid.org/0000-0003-4835-5436\\
Juan Bory-Reyes: https://orcid.org/0000-0002-7004-1794\\
Irene Sabadini: https://orcid.org/0000-0002-9930-4308


\begin{thebibliography}{10}
\bibitem{ACS}  Alpay, D., Colombo, F., Sabadini, I. \textit{Slice hyperholomorphic Schur analysis}. Operator Theory: Advances and Applications, 256. Birkh\"user/Springer, Cham, 2016.

\bibitem{Ba} Baleanu, D., Restrepo, J. E., Suragan, D. \textit{A class of time-fractional Dirac type operators}. Chaos Solitons Fractals 143 (2021), 110590, 15 pp.

\bibitem{Be} Bernstein, S. \textit{A fractional Dirac operator}. Noncommutative analysis, operator theory and applications, 27-41, Oper. Theory Adv. Appl., 252, Linear Oper. Linear Syst., Birkhäuser/Springer, [Cham], 2016.

\bibitem{CC} Cerroni, C. \textit{From the theory of \lq \lq congeneric surd equations" to \lq \lq Segre's bicomplex numbers"}. Historia Math. 44 (2017), no. 3, 232-251.

\bibitem{CG} Contharteze Grigoletto, E., Capelas de Oliveira, E. \textit{Fractional Versions of the Fundamental Theorem of Calculus}. Appl. Math. 4 (2013), 23-33.

\bibitem{CTOP} Coloma, N., Di Teodoro, A., Ochoa-Tocachi, D., Ponce, F. \textit{Fractional Elementary Bicomplex Functions in the Riemann–Liouville Sense}. Adv. Appl. Clifford Alg. 31 (2021), no. 4, Paper No. 63, 29 pp.

\bibitem{CG1} Colombo, F., Gantner, J. \textit{ Fractional powers of quaternionic operators and Kato's formula using slice hyperholomorphicity}. Trans. Amer. Math. Soc. 370 (2018), no. 2, 1045--1100.

\bibitem{CG2} Colombo, F., Gantner, J. \textit{  Quaternionic closed operators, fractional powers and fractional diffusion processes}. Operator Theory: Advances and Applications, 274. Birkh\"auser/Springer, Cham, 2019.

\bibitem{CSS}  Colombo F., Sabadini I., Struppa  D.C. \textit{Noncommutative Functional Calculus. Theory and Applications of Slice Hyperholomorphic Functions}, Birkhauser, Basel, {\bf 289} 2011.

\bibitem{CSS2} Colombo F., Sabadini I.,   Struppa  D.C.  \textit{Entire slice regular functions}. Springer Briefs in Mathematics,  Springer, 2016.

\bibitem{DM} Delgado, B. B., Macías-Díaz, J.E. \textit{On the General Solutions of Some Non-Homogeneous Div-Curl Systems with Riemann–Liouville and Caputo Fractional Derivatives}. Fractal Fract. 2021, 5 (3), 117.

\bibitem{DJRS} Dou X., Jin M., Ren G., Sabadini I. \textit{A new approach to slice analysis via slice topology},
Adv. Appl. Clifford Alg. 31 (2021), no. 5, Paper No. 67, 36.

\bibitem{DRS} Dou X., Ren G., Sabadini I. \textit{Extension theorem and representation formula in non-axially symmetric domains for slice regular functions}, to appear in J. Eur. Math. Soc.

\bibitem{FRV} Ferreira, M,  Krausshar, R. S.,  Rodrigues, M. M., Vieira, N. \textit{A higher dimensional fractional Borel-Pompeiu formula and a related hypercomplex fractional operator calculus}. Math. Methods Appl. Sci. 42 (2019), no. 10, 3633--3653.

\bibitem{FV1} Ferreira, M., Vieira, N. \textit{Eigenfunctions and fundamental solutions of the fractional Laplace and Dirac operators: the Riemann–Liouville case}. Complex Anal. Oper. Theory 10 (5), 1081-1100 (2016).

\bibitem{FV2} Ferreira, M., Vieira, N. \textit{Eigenfunctions and fundamental solutions of the fractional Laplace and Dirac operators using Caputo derivatives}. Complex Var. Elliptic Equ. 62 (9), 1237-1253, (2017).

\bibitem{GenS} Gentili G., Struppa D.C. \textit{A new approach to Cullen-regular functions of a quaternionic variable}, C. R. Math. Acad. Sci. Paris
342 (2006), 741-744.

\bibitem{GenS2} Gentili G., Struppa D.C. \textit{A new theory of regular function of a quaternionic variable}, Adv. Math. 216, 279-301 (2007).
\bibitem{GSto} Gentili, G., Stoppato, C. \textit{A local representation formula for quaternionic slice regular functions}.

Proc. Amer. Math. Soc. 149 (2021), no. 5, 2025--2034.

\bibitem{GenSS} Gentili G., Stoppato C.,  Struppa D.C. \textit{Regular functions of a quaternionic variable}, Springer Monographs in Mathematics. Springer, Heidelberg, 2013.

\bibitem{GB-1} Gonz\'alez Cervantes, J. O., Bory-Reyes, J. \textit{A quaternionic fractional Borel-Pompeiu type formula}, Fractal, Vol. 30, No. 1 (2022) 2250013 (15 pages).

\bibitem{GB-2} Gonz\'alez-Cervantes, J. O., Bory-Reyes, J. \textit{A bicomplex $(\vartheta,\varphi)-$weighted fractional Borel-Pompeiu type formula}.
J. Math. Anal. Appl., 520, No. 2, 2023, 126923 (18 pages).

\bibitem{KV} K\"ahler, U., Vieira, N. \textit{Fractional Clifford analysis}. In: Bernstein, S., K\"ahler, U., Sabadini, I., Sommen, F. (eds.) Hypercomplex analysis: new perspectives and applications. Trends in mathematics, 191-201. Birk\"ahuser, Basel, 2014.

\bibitem{KST} Kilbas, A. A., Srivastava, H. M., Trujillo, J. J. \textit{Theory and Applications of Fractional Differential Equations}. North-Holland Mathematics Studies, 204. Elsevier Science B.V., Amsterdam, 2006.

\bibitem{Le} Leibniz, G. W. \textit{Mathematische Schriften: aus den Handschriften der K\"oniglichen Bibliothek zu Hannover. Briefwechsel zwischen Leibniz, Wallis, Varignon, Guido Grandi, Zendrini, Hermann und Freiherrn von Tschirnhaus}. Druck und Verlag von H.W. Schmidt. Halle, Germany., Volume 1, 1859.

\bibitem{L} Liouville, J. \textit{M\'emoire sur le calcul des diff\'erentielles a indices quelconques}. J. Ecole Polytech., 13, 1832, 71-162.

\bibitem{MR} Miller, K. S., Ross, B. \textit{An Introduction to the Fractional Calculus and Fractional Differential Equations}. A Wiley Interscience Publication. John Wiley \& Sons, Inc., New York, 1993.

\bibitem{OS} Oldham, K. B., Spanier, J. \textit{The Fractional Calculus}. Dover Publ. Inc., 2006.

\bibitem{O} Ortigueira, M. D. \textit{Fractional calculus for scientists and engineers}. Lecture Notes in Electrical Engineering, 84. Springer, Dordrecht, 2011.

\bibitem{PBBB} Peña Pérez, Y. Abreu Blaya, R. Árciga Alejandre, M. P., Bory Reyes, J. \textit{Biquaternionic reformulation of a fractional monochromatic Maxwell system}. Adv. High Energy Phys. 2020, Art. ID 6894580, 9 pp.

\bibitem{P} Podlubny, I. \textit{Fractional differential equations. An introduction to fractional derivatives, fractional differential equations, to methods of their solution and some of their applications}. Mathematics in Science and Engineering, 198. Academic Press, Inc., San Diego, CA, 1999.

\bibitem{R} Riemann, B. \textit{Versuch einer allgemeinen Auffassung der Integration und Differentiation (An attempt to a general understanding of integration and differentiation) (1847)} in: H. Weber (ed.), Bernhard Riemanns gesammelte mathematische Werke und wissenschaftlicher Nachlass, Dover Publications (1953), 353.

\bibitem{Ro} Ross B. \textit{A brief history and exposition of the fundamental theory of fractional calculus}. In: Ross B. (eds) Fractional Calculus and Its Applications. Lecture Notes in Mathematics, vol 457. Springer, Berlin, Heidelberg, 1975.

\bibitem{SKM} Samko, S.G., Kilbas, A. A., Marichev,  O.I. \textit{Fractional Integrals and Derivatives. Theory and Applications}. Gordon and Breach Sci. Publ. London, New York, 1993.

\bibitem{Tarasov} Tarasov, V. E. \textit{No violation of the Leibniz rule. No {VTRMB} fractional derivative}. Commun. Nonlinear Sci. Numer. Simul. 18 (2013), no. 11, 2945-2948.

\bibitem{VTRMB} Valério, D., Trujillo, J. J., Rivero, M., Machado, J. A. T., \& Baleanu, D. \textit{Fractional calculus: A survey of useful formulas}. Eur. Phys. J. Spec. Top. 222, 1827–1846, 2013.

\bibitem{V} Vieira, N. \textit{Fischer decomposition and Cauchy-Kovalevskaya extension in fractional Clifford analysis: the Riemann-Liouville case}. Proc. Edinb. Math. Soc. II. 60 (1), 251-272, 2017.
\end{thebibliography}
\end{document}